\documentclass[11pt]{amsart}
 \usepackage[backref=page]{hyperref}
\hypersetup{
  colorlinks=true, 
   linkcolor=red,  
}

\usepackage{filecontents}

\usepackage{tikz}		

\theoremstyle{plain}
 
\newtheorem{thm}{Theorem} 
\newtheorem{theorem}[thm]{Theorem}
\newtheorem{proposition}[thm]{Proposition} 
\newtheorem{remark}[thm]{Remark} 
\newtheorem{lemma}[thm]{Lemma}
\newtheorem{corollary}[thm]{Corollary} 
\newtheorem{observation}[thm]{Observation} 
\newtheorem{question}[thm]{Open Problem} 
\newtheorem{example}[thm]{Example} 
 
 \numberwithin{thm}{section}

\def\e{\mathbb{E}\, }

\def\p{\mathbb{P} }
\def\P{\mathbb{P} }
\def\BN{\mathbb{N} }

\def\BZ{\mathbb{Z} }
\def\BF{\mathbb{F} }

\def\dtv{d_{\rm TV}}

\def\cQ{\mathcal{Q}}

\def\({\left(}
\def\){\right)}

\newcommand{\ignore}[1]{ }

\def \bZ  {{\mathbf{Z}}}   
\def\bX{{\mathbf{X}}}
\def\bZn{{\mathbf{Z}^{(n)}}}
\def\bXn{{ \mathbf{X}^{(n)} }}  
\def\Xn{X^{(n)} } 

\def\bx {{\bf x}}
\def\bz {{\bf z}}
\def\bv {{\bf v}}
\def\bzero {{\bf 0}}

\def\.{\hskip.06cm}

\title[Countdown for corank, from $\infty$ to zero]{A countdown process, with application to the rank of matrices over $\BF_q(n)$ }  

\author{Richard Arratia and Michael Earnest}
\date{May 11, 2016  7:54  pm}                                       

\begin{document}
\begin{abstract}
 
 Motivated by the work of Fulman and Goldstein,  comparing the distribution of the corank of random matrices in
 $\BF_q[n]$ with the limit distribution as $n \to \infty$,  we define a countdown process,  driven by independent 
 geometric random variables related to random integer partitions.   Analysis of this process leads to sharper bounds on the total variation distance.
 
\end{abstract}

\maketitle

\tableofcontents

\section{Introduction}\label{sect intro}

Fulman and Goldstein \cite{FG}  used Stein's method to get lower and upper bounds on the total variation distance, between the rank distribution for random $n$ by $n+m$ matrices over the finite field $\BF_q$, and its 
 limit for $n \to \infty$.  For $m \ge 0$, with notation that suppresses the dependence on $m$ from the random rank, \cite{FG} proved that the  distance
satisfies the upper and lower bounds
\begin{equation}\label{quote FG}
\frac{1}{8q^{m+n+1}} \le \dtv(\cQ_{q,n}, \cQ_q)   \le \frac{3}{q^{m+n+1}},
\end{equation}
so that the upper bound is  24 times the lower bound.  We provide  a sharper upper bound with a very simple proof in 
Theorem \ref{thm easy 1}.  With more computation,   Theorem \ref{thm dtv corank} provides matching upper and lower bounds, for $m \ge 0$, which differ by a factor of 2. 
Theorem \ref{thm dtv corank} also gives
 an  explicit asymptotic formula, for all $m \in \BZ$. 
 It is necessary to deal with $m<0$ as a separate case from $m \ge 0$;  although matrix transpose provides a handle on the rank distribution, there is a subtle effect on total variation distance,  so that for $m<0$ the lower bound is of order $1/q^{n+1}$ rather than  $1/q^{m+n+1}$.  Note that \cite{FG} also used Stein's method to handle five \emph{other} classes of matrices:
symmetric, symmetric with zero diagonal, skew symmetric, skew centrosymmetric, and Hermitian matrices,  but our method
only handles the simplest case.

\noindent  {\bf Acknowledgement.} We thank Jason Fulman, Larry Goldstein, and Dennis Stanton for helpful conversations.

\section{A Markov chain from linear algebra}\label{sect motivation}

Write $\BF_q$ for the finite field with $q$ elements.   The well-known formula for the number of nonsingular $n$ by $n$ matrices over $\BF_q$, 
 that
$$
| {\rm GL}(n,q)| =  (q^n-1) (q^n-q) \cdots (q^n-q^{n-2}) (q^n - q^{n-1})
$$
has a well-known, and somewhat prettier \emph{probabilistic} interpretation,  by comparing with the number of all $n$ by $n$ matrices over $\BF_q$,
\begin{equation}\label{p nonsingular}
\p(\text{\small{nonsingular}}) = \frac{| {\rm GL}(n,q)|}{|\BF_q[n]|} =\frac{| {\rm GL}(n,q)|}{q^{n^2}} = 
g_n(q^{-1})
\end{equation}
where 
\begin{equation}\label{g n}
     g_n(x) := (1-x^n)(1-x^{n-1})\dots(1-x^2)(1-x).
\end{equation}
This function $g_n$ may be viewed as a perturbation of a  simpler object, the Euler function
\begin{equation}\label{g}
g(x) := \prod_{i \ge 1} (1-x^i), \ \ \text{ for }  x \text{ with }  |x|<1,
\end{equation}
sometimes called the reciprocal of the partition function, and famous for its role in the Euler pentagonal number theorem
\cite{andrews}.

Implicit in \eqref{p nonsingular}  is a story for $n$ by $n+m$ matrices, allowing $m \in \BZ$ to be negative  
but requiring both $n \ge 0$ and $n+m \ge 0$, thinking of $m$ as \emph{time}.    In this story, one thinks about an entire process, evolving in time,  and a natural question arises:  what is the time to \emph{hit zero}, that is,  how many length $n$ columns are needed to span a space of dimension $n$?  
  The process story is given in detail in the following paragraph.

For fixed $n$,  consider independent random vectors $\bv_1,\bv_2,\ldots$,  distributed uniformly over the $q^n$
values in $\BF_q^n$.  With $A_k$ taken to be the space spanned by the first $k$ of these vectors,  so that $A_0$
is the singleton set containing only the all zero vector,  consider the corank of $A_k$,  for $k=0,1,2,\ldots$.  (We say \emph{corank},  thinking of the $n$ by $k$ matrix with columns $\bv_1,\ldots,\bv_k$;  the term \emph{codimension}  might be more correct, but no confusion arises from using the simpler word.)
As $k$ increases,  this  corank decreases from  $n$  down to zero. Given that the corank of $A_k$ is $i$,  and otherwise \emph{independent} of $\bv_1,\bv_2,\ldots,\bv_k$, the chance that
$\bv_{k+1} \in A_k$ so that corank($A_{k+1})=$ corank($A_k)=i$  rather than  corank($A_{k+1})=i-1$, is exactly  $q^{n-i}/q^n = q^{-i}$, \emph{regardless of the value} of $n$.   Trivially, this conditional independence leads to a  
Markov chain, which is  a pure death process,  with independent, geometrically distributed holding times.   We celebrate these observations as a formal statement, for future reference.

\begin{proposition}\label{prop linear algebra}
For any $n \ge 1$,  in the preceeding story over $\BF_q$, write $Y_k := n $ $- $ the dimension of $A_k$.  
Then, with $x := 1/q$,  $Y_0,Y_1,Y_2,\ldots$ is a  Markov chain on $\BZ_+$,  with transition probabilities
\begin{equation}\label{transition x}
   p(i,j) = \left\{ \begin{array}{cl}
                           x^i   &  \text{if } j=i \\
                           1-x^i   &  \text{if } j=i-1 \\
\end{array}   \right. ,
\end{equation}
 starting at $n$.
\end{proposition}
\begin{proof}
The proof is given by the previous paragraph.
\end{proof}

For the sake of comparing the distribution for $n$  with its limit distribution as $n \to \infty$,  it is convenient and natural to shift the time,  replacing  $k$  by $t = k-n$,   so that the growing spaces $A_0,A_1,A_2,\ldots$ have coranks decreasing,  from $n$ down to zero,  with the \emph{time-shift}  taken so that, in the case corresponding to a nonsingular matrix,   corank zero  is hit at \emph{time} $t=0$.
    
\subsection{Counting down from infinity}

The (deterministic)   countdown process, with all zero delays,  is
\begin{equation}\label{zero delay}
\bx \equiv  (x_t)_{t \in \BZ} := \phi(\bzero)  \ \text{ with }  x_{-t}=t,  x_t=0 \text{ for } t=0,1,2,\ldots .
\end{equation}
The space of allowable delays is
\begin{equation}\label{Omega}
   \Omega := \{   \bz = (z_1,z_2,\ldots):   z_1+z_2+\ldots < \infty \} 
   \subset  (\BZ_+)^\BN,
\end{equation}
with least element  $\bzero := (0,0,\ldots) \in \Omega$.
For general $\bz \in \Omega$,  the value $\bx = \phi(\bz)$ of the deterministic countdown process is that perturbation of the path given by
\eqref{zero delay}  such that
\begin{equation}\label{informal}
   z_i  \text{ is the delay at height } i,  \ i=1,2,\ldots, \ \text{ and } 0 = \lim_{t \to \infty}   x_{-t} - t .
\end{equation}    
We use  the indicator notation $1(Q)=1$ if  statement $Q$ is true, $1(Q)=0$ if  statement $Q$ is false.  We also write $\BN = \{1,2,\ldots\} $ and $\BZ_+=\{0,1,2,\ldots\}$.
A formal version of the informal specification \eqref{informal}, naming the domain and codomain,  
 is that
$$ 
\phi:    \Omega \to (\BZ_+)^\BZ
$$
satisfies
\begin{equation}\label{formal}
   \bx = \phi(\bz) \text{ satisifies }  \forall i \ge 1,  1+z_i =\sum_{t \in \BZ} 1(x_t=i), \
   \forall t \,  x_t - x_{t-1} \in \{0,1\}, 
\end{equation}    
$$
    \ \text{ and } 0 = \lim_{t \to \infty}   (x_{-t} - t) \  = \lim_{t \to \infty} x_t.  
$$  
Clearly, the map $\phi$ is a bijection between $\Omega$, and the image,  $\phi(\Omega)$.

\begin{observation}\label{time to hit 0}
Suppose that $\bx=(x_t)_{t\in \BZ}=\phi(\bz)$, where $ \bz \in \Omega$. 
Then the hitting time to zero, for the trajectory $\bx$, is 
$$
 h_0(\bx) := \min\{ t :  x_t=0 \} = z_1+z_2+ \dots \ .
 $$

\end{observation}

Figure 1 shows an example of the process $\phi(\bz)$, when $\bz=(1,3,0,2,0,0\dots)$. The process stays on the line $x=-t$ for all $t\le -4$. The largest $i$ for which $z_i$ is nonzero is $i=4$ with $z_4=2$, so there are two delays at height four, and therefore three points $(t,x_t)$ for which $x_t=4$. Similarly, $z_2=3$ causes the process to be delayed 3 times at height 2, and $z_1=1$ causes $\bx$ to spend one extra unit of time at height 1, before dropping down permanently to the $t$-axis. 

Consider the two circled points in Figure 1 at $(3,2)$ and $(4,1)$. Here, a ``death" has occured at time 3, and $x_t$ has decreased as $t$ increased. Both of these points are on the line $x=-t+5$, whereas the process started on the line $x=-t$. The 
process has moved from the line $x=-t$ for all sufficiently large $x$ to the line $x=-t+5$ for $x=2,1$ 
because there were $5$ delays at heights 2 and above, corresponding to the fact that $z_2+z_3+z_4+\dots=5$. In general, we have the following observation, which will be important later.

\begin{observation}\label{deathandtailsums}
Suppose that $\bx=(x_t)_{t\in \BZ}=\phi(\bz)$, where $\bz=(z_1,z_2,\dots,)$. Then
$$
x_t=k \text{ and } x_{t+1}=k-1\ \  \text{ if and only if }\ \  t+k=z_k+z_{k+1}+z_{k+2}+\dots.
$$

\end{observation}

\def \delay[#1,#2,#3] 
	{
		\foreach \x in {0,...,#3} { \fill (#1+\x,#2) circle (0.1); }
	 }
\begin{figure}
	\begin{tikzpicture}[scale=0.6]
		\draw[opacity=0.2] (-6,0) grid (8,6);
		\draw[color=black,->] (0,0) -- (0,7);
		\draw[color=black,<->] (-7,0) -- (9,0);
		\draw (9,0) node[anchor = north west] {$t$};
		\draw (0,6) node[anchor = south west] {$x$};
		\delay[-6,6,0]
		\delay[-5,5,0]
		\delay[-4,4,2]
		\delay[-1,3,0]
		\delay[0,2,3]
		\delay[4,1,1]
		\delay[6,0,2]
		\draw[dashed][opacity=0.5](-2,7) -- (6,-1);
		\draw[dashed][opacity=0.5] (-7,7) -- (1,-1);
		\draw  (-2,2) node[anchor = north east] {$x=-t$};
		\draw  (2,3) node[anchor = south west] {$x=-t+5$};
		\draw  (3,2) circle (0.2);
		\draw  (4,1) circle (0.2);
	\end{tikzpicture}
	\caption{Countdown process $\phi(\bz)$ when $\bz=(1,3,0,2,0,0,\dots)$.}
\end{figure}
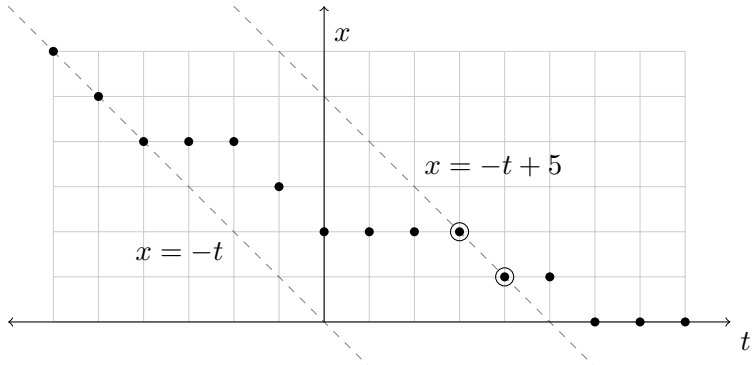

\def \delay[#1,#2,#3] 
	{
		\foreach \x in {0,...,#3} { \fill (#1+\x,#2) circle (0.1); }
	 }
\begin{figure}
	\begin{tikzpicture}[scale=0.6]
		\draw[opacity=0.2] (-6,0) grid (8,6);
		\draw[color=black,->] (0,0) -- (0,7);
		\draw[color=black,<->] (-7,0) -- (10,0);
		\draw (9,0) node[anchor = north west] {$t$};
		\draw (0,6) node[anchor = south west] {$x$};
		\delay[-5,6,0]
		\delay[-4,5,0]
		\delay[-3,4,2]
		\delay[0,3,0]
		\delay[1,2,3]
		\delay[5,1,1]
		\delay[7,0,2]
		\draw[dashed][opacity=0.3](-2,7) -- (6,-1);
		\draw[dashed][opacity=0.3] (-7,7) -- (1,-1);
		\draw  (4,2) circle (0.2);
		\draw  (5,1) circle (0.2);
	\end{tikzpicture}
	\caption{Countdown process $\phi(\bz)$ when $\bz=(1,3,0,2,0,0,\dots,0,1,0,0,\dots)$.}
\end{figure}
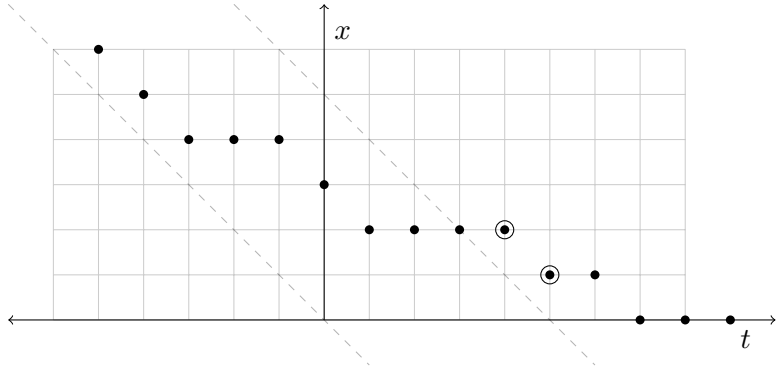

Figures 1 and 2 combined show what will be the ``typical'' difference of interest to this paper.  Looking ahead to 
\eqref{Z}, \eqref{Zn}, and \eqref{X},  in the role of $\bz$ there will be two
choices,  $\bX$ and  $\bXn$, which are very likely identical, but with small probability  differ  by a single 1 in some coordinate with large index $i_0$, and with extremely tiny probability differ in some more violent way.   So we view one of Figures 1,2  as a perturbation of the other.  The
preferred view is that Figure 2, showing the $\bz$ with the extra +1,  corresponds to the \emph{simpler} object, which is the process $\bX$, and that Figure 1, showing $\bXn$, is a perturbation of Figure 2.  This perturbation is  a left-shift by 1.  The two figures show the same window,
which excludes the $t$-neighborhoods of minus and plus infinity,  so that the causal difference of 1 at some $z_{i_0}$ for very large $i_0$
shows up only as this shift;  the jump from the line $x+t=0$  to the line $x+t=1$,  at height $x=i_0$ and times $t=i_0,i_0+1$,  does not show up in the frame of the picture.

\subsection{Geometrically distributed delays, or \emph{driving noise}.}

Fix  $x \in (0,1)$.   Let $\bZ$ 
be a process of \emph{independent} geometrically distributed random variables, with 
\begin{equation}\label{Z}
\bZ = (Z_1,Z_2,\ldots),  \ \p(Z_i \ge k) 
 = x^{ik}, \ k=0,1,2,\ldots .
\end{equation}
(This process is natural to the study of random integer partitions;  
see Remark \ref{partition remark} for some details.)
For the process with all coordinates indexed by $i>n$ zeroed out, we write
\begin{equation}\label{Zn}
\bZn = (Z_1,Z_2,\ldots,Z_n,0,0,\ldots).
\end{equation}

Taking the sum of all coordinates, in each of the two processes specified by \eqref{Z} and \eqref{Zn}, we have
\begin{equation}\label{S}
  S := Z_1+Z_2+ \cdots,  \ \ S_n := Z_1 + \cdots+ Z_n, 
\end{equation}
with $S_0:=0$.
Applying the countdown function $\phi$,  defined by \eqref{informal} -- \eqref{formal}, to each of the two processes  specified by \eqref{Z} and \eqref{Zn},  we have
\begin{equation}\label{X}
  \bX := \phi(\bZ),  \ \bXn := \phi(\bZn).
\end{equation}
We will be interested in comparing $\bX$ with $\bXn$, and a first step is to compare $S$ with $S_n$,  so we also define
 \begin{equation}\label{R}
  R_n := S - S_n = Z_{n+1}+Z_{n+2} + \cdots,
  \end{equation}
for $n=0,1,2,\ldots$.  Note that $S=R_0$.

\begin{proposition}\label{prop X}.
For any $x \in (0,1)$,  for any  $n \ge 1$,   $\Xn_{-n}, \Xn_{-n+1}, \Xn_{-n+2},\ldots$   is a  Markov chain on $\BZ_+$,   with transition probabilities 
given by \eqref{transition x}, starting at $n$.  
\end{proposition}
\begin{proof}
Obvious;  it corresponds to the ``memoryless" property of geometric distributions.
\end{proof}

\begin{proposition}\label{prop X q n}
For $x=1/q$ where $q$ is a prime power, for $n \ge 1$,  the segment  $\Xn_{-n}, \Xn_{-n+1}, \Xn_{-n+2},\ldots$   of $\bXn$ is a realization of $Y_0,Y_1,Y_2,\ldots$, the Markov chain for linear algebra over $\BF_q$ as in Proposition \ref{prop linear algebra}.   In particular,  for integers $m,n$ with $n \ge 1$ and $n+m \ge 0$,  
\begin{eqnarray}
\Xn_m &=^d &  Y_{n+m} \nonumber \\
& =^d & n - \text{the rank of a random $n$ by $n+m$ matrix $M$ over $\BF_q$}. 
\label{q connection n}
\end{eqnarray}
\end{proposition}
\begin{proof}
Obvious again,  apart from the trickiness of the time shift by $n$ connecting the two processes.
\end{proof}

\begin{proposition}\label{prop X q}
 For $x=1/q$ where $q$ is a prime power, for  $m \in \BZ$,   $X_m$,  from the process $\bX$ defined by \eqref{Z} and \eqref{X}, is distributed as the limit, upon $n \to \infty$, of $n$ minus 
 the rank of a random $n$ by $n+m$ matrix $M$ over $\BF_q$.
 \end{proposition}
\begin{proof}
Obvious from Proposition \ref{prop X q n} and the coupling, with  $\p(X_m \ne \Xn_m) \le \p(\bX \ne \bXn) = \p(\bZ \ne \bZn) = \p(Z_{n+1}+Z_{n+2}+\dots \ne 0)  $.
\end{proof}

\section{Easy bounds on total variation distance}

The total variation distance between random elements, say $X,Y$ in a space $S$ is defined, in general,  by
\begin{equation}\label{dtv 1}
\dtv(X,Y)= \sup_{B \subset S} \ | \p(X\in B) - \p(Y \in B)|,
\end{equation}
where the supremum is taken over \emph{measurable} subsets of $S$.
In case $S$ is discrete,  this is equivalent to
\begin{equation}\label{dtv 2}
\dtv(X,Y)= \sum_k   \max(0, \p(X = k) - \p(Y = k)).
\end{equation}
Another characterization of total variation distance is that  
  the total variation distance  between $X$ and $Y$ is equal to the infimum, over all couplings, of $\p(X\neq Y)$;  it is understood that the \emph{marginal} distributions of $X$ and $Y$ have been specified, and \emph{coupling}  means to choose \emph{any joint distribution} for $(X,Y)$ having the given marginals.
 
\begin{theorem}\label{thm dtv process} 
For   any $n \ge 1$ and  $x \in (0,1)$,  the  total variation distance between the
processes $\bX$  and  $\bXn$  defined by \eqref{Z} and \eqref{X}  is
 \begin{eqnarray}
 \dtv(\bX,\bXn) & = & \p(R_n>0) \nonumber  \\
 & = &  1-\prod_{i=n+1}^\infty (1-x^i)    \label{first upper bound},
\end{eqnarray}
and
\begin{equation}\label{bounds dtv process}
\frac{x^{n+1}}{1-x}-\frac{x^{2n+3}}{(1-x)^2}\le d_{TV}(\bX,\bXn)\le \frac{x^{n+1}}{1-x}.
\end{equation}¥

\end{theorem}
\begin{proof}  Let $A=\{\bx\in (\BZ_+)^\BZ: x_t\neq -t \text{ for some } t<-n\}$ be the set of paths in which  a delay happened at some height above $n$. Then $\p(\bXn\in A)=0$, but $P(\bX\in A)=\p(Z_{n+1}+Z_{n+2}+\dots>0)=\p(R_n>0),$ showing that $$d_{TV}(\bX,\bXn)\ge \p(\bX\in A)-\p(\bXn\in A)=\p(R_n>0).$$ 

Conversely, the total variation distance  between $\bX$ and $\bXn$ is equal to the infimum, over all couplings, of $P(\bX\neq \bXn)$. In the coupling given by \eqref{X}, $\bX$ and $\bXn$ are unequal if and only if $Z_{n+1}+Z_{n+2}+\dots>0$, proving the corresponding upper bound, and hence the equality in \eqref{first upper bound}.

The lower and upper bounds in \eqref{bounds dtv process} follow from the exact expression on the right side 
of \eqref{first upper bound} and from the ``Bonferroni" inequalities
\begin{equation}
\sum_{1\le i\le m} p_i-\sum_{1\le i<j\le m}  p_i p_j \ \le \	1-\prod_{i=1}^m (1-p_i) \le \sum_{1\le i\le m}  p_i \label{bonferroni}
\end{equation}
valid for all $m\in \BN$, provided $0\le p_i\le 1$ for $1\le i\le m$. To use these to prove  \eqref{bounds dtv process}, set $p_i=x^{n+i}$, let $m\to\infty$, then use the facts that $\sum_{i> n}x^i=x^{n+1}/(1-x)$ and $\sum_{j>i>n}x^ix^j = x^{2n+3}/(1-x)^2$.

The classical (first two)  Bonferroni inequalities are
$$
  \sum_{1 \le i \le m} \p(B_i)  - \sum_{1 \le i < j \le m} \p(B_i \cap B_j)   \le   \p \left( \bigcup B_i \right) \le  \sum_{1 \le i \le m} \p(B_i),
$$
and with $B_i := \{ Z_{n+i} \ge 1\}$ we have $p_i=\p(B_i)$, and for $i<j$,  $p_ip_j =   \p(B_i \cap B_j)$,  using the independence of
$Z_1,Z_2,\dots$,  we have exactly \eqref{bonferroni}.

\ignore{   
The upper bound \eqref{bonferroniupper} has an easy proof by induction on $m$. To prove \eqref{bonferronilower} by induction, note that
\begin{align*}
\prod_{i=1}^m (1-p_i)
	& =(1-p_m)\prod_{i=1}^{m-1}(1-p_i)\\
	&\le (1-p_m)\(1-\sum_{1\le i\le m-1} p_i+\sum_{1\le i<j\le m-1}p_ip_j\)\\
	&\le (1-p_m)\(1-\sum_{1\le i\le m-1} p_i\)+\sum_{1\le i<j\le m-1}p_ip_j\\
	&=1-\sum_{1\le i\le m} p_i+\sum_{1\le i<j\le m}p_ip_j.
\end{align*}¥
}   

\end{proof}
\begin{corollary}\label{cor functional}
For   any $n \ge 1$, $t \in \BZ$,  and  $x \in (0,1)$,
\begin{equation}\label{marginal upper bound}
\dtv( X_t, \Xn_t)  \le u(n) := 1-\prod_{i=n+1}^\infty (1-x^i)  \le \frac{x^{n+1}}{1-x}
\end{equation}
and 
\begin{equation}\label{S upper bound}
\dtv(S, S_n)  \le u(n) := 1-\prod_{i=n+1}^\infty (1-x^i)  \le \frac{x^{n+1}}{1-x}.
\end{equation}

 \end{corollary}
\begin{proof}
These upper bounds are an immediate corollary of Theorem \ref{thm dtv process}, since 
with the deterministic function $e_t$,  ``extract   coordinate $ t$'', we have $X_t=e_t(\bX)$
and $\Xn_t=e_t(\bXn)$,  and with deterministic functional  $h_0$,  the hitting time to zero from Observation \ref{time to hit 0},
we have $S=h_0(\bX)$ and $S_n =h_0(\bXn)$.
\end{proof}

Note the the upper bound $u(n)$ in \eqref{marginal upper bound} \emph{does not vary} with $t \in \BZ$.
In the next proposition,  we give a sort of matching lower bound; this lower bound varies with $t$.  The upper bound in \eqref{marginal upper bound} is quite poor for $t>0$, as will eventually be seen from 
 Theorem \ref{thm dtv corank}. The next proposition gives our ``easy'' lower bound for cases with $t \le 0$.  
\begin{proposition}\label{prop nonpositive t}
For   $x \in (0,1)$,  for integers $n,t$ with  $n \ge 1$, $t \le 0$,    
\begin{eqnarray*}
\dtv( X_t, \Xn_t) & \ge &  \p(\Xn_t+t=0) - \p(X_t+t=0) \\
& = & \left( \prod_{-t < i \le n} (1-x^i) \right) \left(1-\prod_{i=n+1}^\infty (1-x^i)  \right) \\
& \ge & \ell(t,n) := \left( \prod_{-t < i \le n} (1-x^i) \right) \ \left( \frac{x^{n+1}}{1-x}-\frac{x^{2n+3}}{(1-x)^2} \right)  .
\end{eqnarray*}
 \end{proposition}
\begin{proof}
For $t\le 0$,  that  the event $\{X_t+t=0\}$ equals the event that $(Z_{-t+1}=Z_{-t+2}=\dots=0)$, and  $\{\Xn_t+t=0\}$ equals the event that $(Z_{-t+1}=Z_{-t+2}=\dots=Z_n=0)$.  The inequality is the same that we used in getting \eqref{bounds dtv process}
from \eqref{first upper bound}.

\end{proof}

Consider the relation between exact formulas, asymptotics, lower bounds, and upper bounds.  We use  the notation $a_n\sim b_n$ to mean that $a_n$ is \emph{asymtotically equal} to $b_n$,  formally defined by $\lim_{n\to \infty}a_n/b_n=1$.   It is obvious from 
\eqref{bounds dtv process}  that exact expression for the  distance between \emph{processes} may
be described asymptotically, with
$$
d_{TV}(\bX,\bXn)\sim \frac{x^{n+1}}{1-x}.
$$
It is more difficult to give asymptotics for the distance between the marginals, $\dtv( X_t, \Xn_t)$.
 Corollary \ref{cor functional}  and Proposition \ref{prop nonpositive t} provide upper and lower bounds, $u(n)$ and
 $\ell(t,n)$,  with
 $$
 r(t) := \lim_{n \to \infty} \frac{\ell(t,n)}{u(n)} = \prod_{-t < i} (1-x^i) 
 $$
 which, with the notation from \eqref{g} and the display above \eqref{g} is $r(t) = g(x) / g_{-t}(x)$.   At $t=0$ the product $g_{-t}$ has no factors;  it is identically 1,  and we have $r(t)=g(x)$.
 As $t \to -\infty$,  $g_{-t}$ acquires more and more of the factors of $g$,  and $r(t) \nearrow 1$, so in a sense,
 the upper and lower bounds combined come close  to giving the asymptotic total variation distance.
 
    We will completely handle the task of giving asymptotics for $\dtv( X_t, \Xn_t)$, with Theorem   \ref{thm dtv corank}.

\begin{proposition}\label{row col rank}
For $x=1/q$ where $q$ is a prime power,  for any integers $t,n$ with $n \ge 0,  n+t \ge 0$
\begin{equation}\label{shift identity}
\Xn_t =^d X^{(n+t)}_{-t} -t.
\end{equation}
\end{proposition}
\begin{proof}
Consider a random $n$ by $n+t$ matrix $M$ over $\BF_q$.   We exploit the fundamental result that row rank equals column rank.
\begin{eqnarray*}
\Xn_t    & =^d &n - \text{ column rank of } M \\
  & = &  n - \text{ row rank of } M \\
  & = &  (n+t - \text{ column rank of } M^T) - t \\
  & =^d & X^{(n+t)}_{-t} -t.
\end{eqnarray*}  
\end{proof}

\begin{theorem}\label{thm easy 1}
For any prime power $q$, for any  $m \ge 0$,  for every $n \ge 1$,  the total variation distance, between the  $\cQ_{q,n} :=$ ($n$ minus the 
rank of a random $n$ by $n+m$ matrix over $\BF_q$),  and  $\cQ_q := $the distributional limit upon $n \to \infty$ of $\cQ_{q,n} $, satisfies
$$
\dtv(\cQ_{q,n},\cQ_{q})  \le \frac{q}{q-1} \ \frac{1}{q^{n+m+1}}.
$$
\end{theorem}
\begin{proof}
Combine Propositions \ref{prop X q n} and \ref{prop X q}, together   with 
Proposition \ref{row col rank}  and the bound \eqref{marginal upper bound} from Corollary
\ref{cor functional}
applied at $t=m$.  Note that we are using $x=1/q$,  so that $1/(1-x) = q/(q-1)$.
\end{proof}

\begin{theorem}\label{thm easy 2}
Exactly as in Theorem \ref{thm easy 1},  \emph{except} that now we take $m < 0$.  For every  $n$ with
$n+m \ge 0$, 
$$
\dtv(\cQ_{q,n},\cQ_{q})  \le \frac{q}{q-1} \ \frac{1}{q^{n+1}}.
$$
\end{theorem}
\begin{proof}
Just as the proof of Theorem \ref{thm easy 1},  except that we do not invoke
 Proposition \ref{row col rank}.
\end{proof}


 \section{A technical lemma}
 
 For the computations in the next section, we will require the following fact. 
 
 \begin{lemma}\label{technical} For all $n\ge m\ge 0$, and $|x|<1,|y|<1$, 
$$
\prod_{i=m}^n\frac1{1-yx^i}=\sum_{k\ge0} y^kx^{mk}\cdot \frac{\prod_{i=n-m+1}^{n-m+k}1-x^i}{\prod_{i=1}^k1-x^i}.
$$
\begin{proof}
It suffices to prove this in the case $m=0$, namely to show that
\begin{equation}\label{simpler}
\prod_{i=0}^n\frac1{1-yx^i}=\sum_{k\ge0} y^k\cdot \frac{\prod_{i=k+1}^{n+k}1-x^i}{\prod_{i=1}^k1-x^i},
\end{equation}
since the general result follows from replacing $n$ with $n-m$ and $y$ with $yx^m$ in  \eqref{simpler}.

Letting $F_n(x,y)$ denote the right hand side of \eqref{simpler}, elementary power series manipulations obtain that $(1-yx^n)F_n(x,y) = F_{n-1}(x,y)$, which implies
$$
F_n(x,y) =\frac1{1-yx^n}F_{n-1}(x,y).
$$
Iterating the latter relation $n$ times yields that $$F_n(x,y) = \left(\prod_{i=1}^n\frac1{1-yx^i}\right)F_0(x,y),$$ which combined with the base case $F_0(x,y) = \sum_{k\ge0} y^k=\frac1{1-y}$ proves \eqref{simpler}.

\end{proof}
\end{lemma}

\begin{remark}\label{partition remark}
A more conceptual proof of this result can be given, relating the result to random integer partitions, where a partition of $r$ is given weight $x^r$;   see \cite{andrews1974, fristedt,IPARCS, PDC}. (In contrast with the linear algebra applications involving $x=1/q$, bounded away from 1,  taking $x=\exp(-\pi/\sqrt{6r})$ leads to excellent approximations for a random partition of a large  integer $r$;  see \cite{PittelShape}.) In more detail,  $Z_i$  is interpreted as the number of parts of size $i$,  so that $r=\sum i Z_i$ is the size of the partition, and $k= \sum Z_i$ is the number of parts.   We are considering partitions  where all part sizes lie in the range  $m$ to $n$,  and such a partition $\lambda$ of size $r$, with exactly $k$ parts, is in bijective correspondence with a partition $\lambda'$ of $r-km$ with at most $k$ parts, each of size at most $n-m$,
by removing $m$  from each part of $\lambda$.
\end{remark}
 
 \section{Distributional Results}
 
 We compute the distributions of $\Xn_t$, and $X_t$, for all for $x \in (0,1)$, $n \ge 0$, and $t\in \BZ$.
 The explicit formulas lead to an interesting symmetry, stated in Corollary  
 \ref{time reversal symmetry}, which in case $x=1/q$, where $q$ is a prime power, was already proved, in Proposition \ref{row col rank}.

  \begin{lemma}\label{rdist} For all $x \in (0,1)$, for all $k\in \BZ^+$, and $n\ge m> 0$,
 $$
 	\p(S_n - S_{m-1}=k) = x^{mk}\cdot  \frac{\prod_{i=n-m+1}^{n-m+k}(1-x^i)\prod_{i=m}^n (1-x^i)}{\prod_{i=1}^{k}(1-x^i)}.
 $$
 \end{lemma}
\begin{proof} Let $G(s):=\e[s^{(S_n-S_{m-1})}]$ be the probability generating function for $S_n-S_{m-1}$, so $G(s)=\sum_{k=0}^\infty \p(S_n-S_{m-1}=k)s^k$. Since $S_n-S_{m-1}$ is a sum $Z_{m}+Z_{m+1}+\dots+Z_n$ of independent geometric random variables, each with probability generating function $\e[s^{Z_i}]=\frac{1-x^i}{1-sx^i}$, it follows that
$$
	\sum_{k=0}^\infty \p(S_n-S_{m-1}=k)s^k=G(s) = \prod_{i=m}^n \frac{1-x^i}{1-sx^i}=
	\frac{g_n(x)}{g_{m-1}(x)}
	\prod_{i=m}^n \frac1{1-sx^i},
$$
where $g_n$ is given by \eqref{g n}. Using Lemma \ref{technical}, we can rewrite the product on the right hand side of the previous equation as
$$
	\frac{g_n(x)}{g_{m-1}(x)}\prod_{i=m}^n \frac1{1-sx^i} 
	= \frac{g_n(x)}{g_{m-1}(x)}\sum_{k\ge0}s^kx^{mk} \frac{\prod_{i=n-m+1}^{n-m+k}1-x^i}{\prod_{i=1}^{k}1-x^i}.
$$
Finally, combining the last two equations, then equating the coefficients of $s^k$, proves the lemma.
\end{proof}

\begin{corollary}\label{Sn is logconcave} The distribution of $S_n$ is logconcave.
\end{corollary}
\begin{proof}Using Theorem \ref{rdist},  for $k \ge 0$, 
$$
\P(S_n=k)=x^k(1-x^n)\prod_{i=k+1}^{n+k-1}(1-x^i).
$$
Cancellation of some common factors leads to
\begin{equation}\label{S ratio}
\frac{\p(S_n=k+1)}{\p(S_n=k)} = x  \ \frac{1-x^{n+k}}{1-x^{k+1}}.
\end{equation}

Using \eqref{S ratio}, one can verify that
$$
\p(S_n=k+1)^2  \ge \p(S_n=k) \, \p(S_n=k+2)
$$
holds for all $k\in \BZ$, which 
is precisely the condition  that the distribution is log concave.

\end{proof}

\begin{remark}

In the special case  $x=1/q$ where $q$ is a prime power,
the product formula \eqref{product form} for distribution of $\Xn_t$, given below in Theorem \ref{xdist},
governs the number of rectangular matrices of a given rank over $\BF_q$,  and this case can be traced back to 1893 
\cite{landsberg}; it also appears as \cite[p. 157, problem 192b]{stanley}. 

The product formula has an easy combinatorial proof. Recall that the \emph{q-binomial} coefficients, defined by
$$
\binom{n}{k}_q=\prod_{i=0}^{k-1} \frac{1-q^{n-i}}{1-q^{i+1}},
$$
give the number of $k$-dimensional subspaces of $\BF_q^n$.
 Consider an $n\times (n+t)$ matrix with rank $n-k$ as a linear map from $\BF_q^{n+t}$ to $\BF_q^n$. There are $\binom{n+t}{t+k}_q$ choices for the $(t+k)$-dimensional kernel $K$ of this map, $\binom{n}{n-k}_q$ choices for the $(n-k)$-dimensional image, $I$, and there are $\prod_{i=0}^{n-k-1}(q^{n-k}-q^i)$ ways to specify the nonsingular linear transformation from a fixed complement of $K$ to $I$. This proof is due to Dennis Stanton, private communication. 

For symmetric, skew-symmetric, and Hermitian matrices, there are analogous product formulas for the number of matrices of a given rank, due to Carlitz and Hodges,   see \cite[page 661]{stanton}.

It would seem natural that there should be a transfer principle,  so that knowing the result for $x=1/q$ implies the result for all $x \in (0,1)$,   but we don't know of such a principle. 
Such a transfer principle would also apply to reflection symmetry, allowing Proposition \ref{row col rank} to imply Corollary \ref{time reversal symmetry}.
 We 
 believe that our proof of Theorem \ref{xdist}, exploiting the Markov property of the countdown process, has both simplicity and novelty.
\end{remark}

\begin{question}
Is there a transfer principle,  allowing results for the Markov chain defined by \eqref{transition x} with parameter $x \in (0,1)$ to be deduced, with no extra computation,  from the combinatorial and linear algebraic results corresponding to the cases $x=1/q$ where $q$ must be a prime power?
\end{question}
 
 \begin{theorem}\label{xdist}
 For all $x \in (0,1)$, for all $n\ge k \ge 0$, and for all $t\ge -k$,
\begin{equation}
\label{product form}
\p(\Xn_t=k)=x^{k(t+k)}\cdot \frac{\prod_{i=n-k+1}^{n+t}(1-x^i)\prod_{i=k+1}^n (1-x^i)}{\prod_{i=1}^{t+k}(1-x^i)},
\end{equation}
and 
$$
\p(X_t=k) = x^{k(t+k)}\cdot \frac{\prod_{i=k+1}^\infty (1-x^i)}{\prod_{i=1}^{t+k}(1-x^i)}.
$$
 \end{theorem}

\begin{proof}
Recall 
from Proposition \ref{prop X} 
that 
$\{\Xn_t\}_{t=-n}^\infty$ is a time-homogenous Markov process, with transition probabilities
 \begin{equation}\label{transition}
 \p(\Xn_{t+1}=k|\Xn_{t}=k)=x^k,\ \ \p(\Xn_{t+1}=k-1|\Xn_{t}=k)=1-x^k.
 \end{equation}
In case $k \ge 1$, let $D_{t,k}=\{\Xn_t=k,\Xn_{t+1}=k-1\}$ be the event that there is a ``death" at time $t$ and height $k$, and $V_{t,k}=\{\Xn_{t}=k,\Xn_{t+1}=k\}$ be the event of a survival. Provided $k\ge 1$,  \eqref{transition} implies that the ratio of the probabilities of survival to death is given by $\p(V_{t,k})/\p(D_{t,k})=x^k/(1-x^k)$. Combined with the fact that $\{\Xn_{t}=k\}$ is the disjoint union of $D_{t,k}$ and $V_{t,k}$, we get that
\begin{equation}\label{xtod}
\p(\Xn_t=k) 
=\frac{1}{1-x^k}\p(D_{t,k}).
\end{equation}
The reason that \eqref{xtod} is useful comes from observation \ref{deathandtailsums}, which implies that $D_{k,t}$ occurs if and only if $t+k=Z_k+Z_{k+1}+Z_{k+2}+\dots+Z_n$, so 
 \begin{equation}\label{dtor}
 \p(D_{t,k})=\p(S_n - S_{k-1}=t+k).
 \end{equation}¥
Combining \eqref{xtod}, \eqref{dtor} and Lemma  \ref{rdist} proves 
\eqref{product form} 
 for all $k\ge 1$.

In case $k=0$, we must prove that 
\begin{equation}\label{kiszero}
\P(\Xn_t=0)=\prod_{i=t+1}^{n+t}(1-x^i)
\end{equation}
which we prove by induction on $t$. The base case that $P(\Xn_0=0)=\prod_{i=1}^{n}(1-x^i)$ holds since $\Xn_0=0$ exactly when $Z_1=Z_2=\dots=Z_n=0$. Assuming that \eqref{kiszero} holds for $t-1$, note that $\Xn_{t}=0$ implies that $X_{t-1}$ is either 0 or 1. Then 
\begin{align*}
P(\Xn_{t}=0) 
	&= P(\Xn_{t-1}=0) + P(D_{t-1,1})\\
	&= P(\Xn_{t-1}=0) + P(S_n-S_0=t)\\
	&= \prod_{i=t}^{n+t-1}(1-x^i) 
		+ x^t\cdot \frac{\prod_{i=n}^{n+t-1}(1-x^i)\prod_{i=1}^n(1-x^i)}{\prod_{i=1}^t(1-x^i)}\\
	&= \prod_{i=t}^{n+t-1}(1-x^i) +x^t(1-x^n)\prod_{i=t+1}^{n+t-1}(1-x^i)\\
	&= \(\prod_{i=t+1}^{n+t-1}(1-x^i)\) \cdot \Big((1-x^t)+x^t(1-x^n)\Big)\\
	&=\prod_{i=t+1}^{n+t}(1-x^i),
\end{align*}¥
completing the proof by induction.

Finally, $\p(X_t=k) = \lim_{n\to\infty} \p(\Xn_t=k)$, since $\Xn_t$ converges to $X_t$ almost surely, and therefore in distribution. 

\end{proof}

 \begin{corollary}\label{time reversal symmetry}
 For all $x \in (0,1)$, for all $n\ge 0$, for all  $t\in \BZ$, 
$$\Xn_t=^d X^{(n+t)}_{-t}-t,
$$ 
and
$$
X_t=^d X_{-t}-t.
$$
\end{corollary}
\begin{proof}
It is routine to use Theorem \ref{xdist} to verify that for all $k\ge 0$, $\p(X^{(n)}_t=k) = \p(X^{(n+t)}_{-t}=t+k)$ for all $n\ge0$ and $\P(X_t=k)=\P(X_{-t}=t+k)$. 
\end{proof}

 \section{Total Variation Distances, for $0 < x \le 1/2$}

\begin{theorem}\label{thm dtv sum} 
Suppose that $x\in (0,1)$. Then, in case $x \le 1/2$,
\begin{align*}
d_{TV}(S,S_n) 
	&=\prod_{i=1}^n(1-x^i)\(1-\prod_{i=n+1}^\infty (1-x^i)\)\\
	&\sim  \frac{g(x)}{1-x} \cdot x^{n+1},
\end{align*}
where $g(x)$ is defined in \eqref{g}.

\end{theorem}
 \begin{proof}
 From Lemma \ref{rdist}, we have that 
 \begin{align*}
 \P(S_n=k) &= x^k(1-x^n)\prod_{i=k+1}^{n+k-1}(1-x^i),\\
 \P(S=k) &= x^k\cdot \prod_{i=k+1}^\infty (1-x^i).
 \end{align*}
 The key observation is that $P(S_n=k)>P(S=k)$ when $k=0$, but the reverse inequality holds otherwise. To see this, consider the ratio
 $$
 \frac{\P(S_n=k)}{\P(S=k)} = \frac{1-x^n}{\prod_{i=n+k}^\infty (1-x^i)}.
 $$

When $k=0$, this ratio is $\frac{1-x^n}{\prod_{i=n}^\infty (1-x^i)}=\frac{1}{\prod_{i=n+1}^\infty (1-x^i)}>1$, while when $k\ge 1$, the ratio is less than one, as shown below:
$$
1-x^n \le 1-\frac{x^{n+1}}{1-x}=1-\sum_{i=n+1}^\infty x^i\le \prod_{i=n+1}^\infty(1-x^i)\le \prod_{i=n+k}^\infty(1-x^i).
$$
The first inequality above uses the fact that $\frac{x}{1-x}\le 1$, which follows from $x\le 1/2$. 

We have proven that $\P(S_n=k)>\P(S=k)$ when $k=0$, but the reverse inequality holds otherwise, which implies that total variation distance is simply given by
$$
\dtv(S,S_n)=\P(S_n=0)-\P(S=0)=\prod_{i=1}^n(1-x^i)\(1-\prod_{i=n+1}^\infty (1-x^i)\).
$$
The above exact expression for $ \dtv(S,S_n)$ implies the asymptotic result $\dtv(S,S_n)\sim x^{n+1}\cdot  g(x)/(1-x)$ by using the bounds
 $$
\( \sum_{i>n}x^i \)- x^{2n+3}/(1-x)^2\le 1-\prod_{i=n+1}^\infty (1-x^i)\le \sum_{i>n}x^i,
 $$
 which follow from \eqref{bonferroni}. 

 \end{proof}

Finally, we have exact and asymptotic results for the total variation distance between each coordinate $X^{(n)}_t$ and $X_t$ of the two processes, provided that $x\le \frac12$.

\begin{theorem}\label{thm dtv corank} 
Suppose that $x\in (0,1/2]$. For all $n,t\ge0$,
\begin{align*}
\dtv(X_t,  X^{(n)}_t) 
	&= \left(\prod_{i=t+1}^{n+t} 1-x^i\right)\left(1-\prod_{i=n+t+1}^\infty 1-x^i \right)\\
	&\sim \frac{C_t}{1-x}\cdot x^{n+t+1},
\end{align*}
where $C_t=\prod_{i=t+1}^\infty1-x^i$. When $t<0$, 
\begin{align*}
\dtv(X_t,  X^{(n)}_t)    
	&= \left(\prod_{i=|t|+1}^{n} 1-x^i\right)\left(1-\prod_{i=n+1}^\infty 1-x^i \right)\\
	&\sim \frac{C_{|t|}}{1-x}\cdot x^{n+1}.
\end{align*}
\end{theorem}
 \begin{proof}
 The proof is similar to that of Theorem \ref{thm dtv sum}.
 
First, suppose $t\ge 0$. For any $k\in \BZ_+$, consider the ratio between $\P(\Xn_t=k)$ and $\P(X_t=k)$, which we attain using Theorem \ref{xdist}:
\begin{equation}\label{ratio}
\frac{\P(\Xn_t=k)}{\P(X_t=k)}=\frac{\prod_{i=n-k+1}^n(1-x^i)}{\prod_{i=n+t+1}^\infty(1-x^i)}.
\end{equation}¥

When $k=0$, the numerator is an empty product, which means the above ratio is greater than one. However, for all $k\ge 1$, the ratio is less than one, as shown be the following computation:
$$
\prod_{i=n-k+1}^n(1-x^i)
	\le 1-x^n\le 1-\frac{x^{n+1}}{1-x}
	\le 1-\sum_{i=n+t+1}^\infty x^i
	\le \prod_{i=n+t+1}^\infty (1-x^i).
$$
 As in the proof of Theorem \ref{thm dtv sum}, the total variation distance is simply given by
 \begin{align*}
 \dtv(X_t,\Xn_t)
 	&=\P(\Xn_t=0)-\P(X_t=0)\\
	&=\prod_{t+1}^{n+t}(1-x^i) - \prod_{t+1}^\infty (1-x^i)\\
	&= \left(\prod_{t+1}^{n+t} 1-x^i\right)\left(1-\prod_{n+t+1}^\infty 1-x^i \right),
 \end{align*}
 as claimed. The asymptotic result also follows similarly.
   
 In case $t<0$, the result follows from the $t\ge 0$ case by using Corollary \ref{time reversal symmetry}, which implies that
 $$
 (X_{t},\Xn_{t})=^d (X_{-t}+t,X^{(n-t)}_{-t}+t).
 $$
 so 
 $$
 \dtv(X_t,\Xn_{-t})=\dtv(X_{-t}+t,X^{(n-t)}_{-t}+t)=\dtv(X_{|t|},X^{(n-t)}_{|t|}).
 $$
 
 \end{proof}

 We remark that Theorem \ref{thm dtv corank} implies the more elementary bounds,  for $ t \ge 0$
 $$
 \frac{1}{2(1-x)}\cdot x^{n+t+1}\le \dtv(X_t,\Xn_t)\le \frac1{1-x}\cdot x^{n+t+1}.
 $$
In the case $x=1/q$,  this narrows the ratio of upper bound to lower bound in \eqref{quote FG}
from 24 to 2;  see Theorem \ref{thm easy 1} for the notational details of how, with $t=m  \ge 0$,  our
$\Xn_t$ corresponds to $\cQ_{q,n}$ and $X_t$ corresponds to $\cQ_q$.

\section{Total Variation Distances,  allowing $x > 1/2$}

For the application to linear algebra,  one always has $x = 1/q \le 1/2$,  so the results of the previous section are adequate.
For the countdown process in general, it is possible to analyze the asymptotic total variation distance,   for  the hitting time to zero, i.e., $S$ versus $S_n$;   Theorem \ref{thm dtv sum 2} below contains Theorem \ref{thm dtv sum} as a special case.  For the analysis of the asymptotic total variation distance for the height at time $t$, i.e.,  $X_t$ versus $\Xn_t$,  there are further obstacles, and we don't have a generalization for Theorem \ref{thm dtv corank}.  Nevertheless, we believe that the point of view given in the paragraph following Observation \ref{deathandtailsums} could be the starting point for such analysis.

\begin{question}
Give asymptotics for $\dtv(X_t,\Xn_t)$,   as $n \to \infty$, for fixed $t \in \BZ$, and allowing $0 < x < 1$.
\end{question}

For our analysis of the asymptotic value of $\dtv(S,S_n)$, we begin with a few general principles, in the form of Lemma \ref{dtv unimodal shift},  Lemma \ref{mixture},  and
Corollary \ref{dtv var plus bernoulli}.   After that, we give a bit of concrete calculation in Proposition \ref{Rn tail},  and Theorem \ref{thm dtv sum 2} follows easily.

An integer-valued random variable $X$ is said to be \emph{unimodal}  if there exists a value $k_0$,
such that  $\p(X=i) \le \p(X=i+1)$ for  $i<k_0$ and $\p(X=i) \ge \p(X=i+1)$  for $i \ge k_0$.   In this case, we say that the distribution of $X$ is unimodal, with mode at $k_0$.  It is a standard fact, easily proved, that if a distribution is \emph{log concave},  then it must be \emph{unimodal}.

\begin{lemma}\label{dtv unimodal shift}
Let $X$ be an integer valued random variable whose distribution is unimodal. Then
$$\dtv(X,X+1)=\sup_{k\in \BZ}\P(X=k).
$$
\end{lemma}
\begin{proof}
Starting from \eqref{dtv 2},  we have 
$$ 
\dtv(X,X+1) 
 = \sum_i   \max(0, \p(X = i) - \p(X = i-1)) .
$$
When the distribution of $Y$ is unimodal with mode at $k_0$, the above simplifies to
$$ 
\dtv(X,X+1)  =  \sum_{i \le k_0}  \p(X = i) - \p(X = i-1)
$$
and the sum telescopes to give  $\dtv(X,X+1)  =  \p(X= k_0)$. 
Obviously, unimodality implies that 
 $\sup_{k\in \BZ}  \p(X=k) = \P(X=k_0)$.  
\end{proof}

\begin{example}
Suppose  $X$ is Poisson with mean $\lambda \in (0,\infty)$.   The distribution of $X$ is unimodal, with mode at $k_0 = \lfloor \lambda \rfloor$,   i.e.,  $\lambda$ rounded down to an integer. We have
$\dtv(X, X+1) = \p(X=k_0)$.  By Stirling's formula, as $\lambda \to \infty$, $\dtv(X,X+1) \sim 1/\sqrt{2 \, \pi \,  \lambda}$.
\end{example}

\begin{lemma}\label{mixture}
Suppose that the distribution of $Y$ is a mixture of the distributions of $X$ and $Z$,  with
$\p(Y \in B) = (1-p) \p(X \in B) + p \, \p(Z \in B)$ for all measurable $B$.  Then 
$$
\dtv(X,Y) = p \ \dtv(X,Z).
$$
\end{lemma}
\begin{proof}
Obvious from \eqref{dtv 1}.
\end{proof}
   
\begin{corollary}\label{dtv var plus bernoulli}
Let $X$ be an integer valued random variable,  whose distribution is unimodal, and let $U$ be independent of $X$, with the Bernoulli distribution having parameter $p$,  i.e.,  $\p(U=1)=p=1-\p(U=0)$.  Then  
$$\dtv(X,X+U)= p \ \sup_{k\in \BZ}\P(X=k).$$
\end{corollary}
\begin{proof}
Lemma \ref{mixture} applies here, with $Y=X+U$ having distribution a mixture of the distributions of $X$ and $Z=X+1$.
\end{proof}

\begin{example}
Suppose  $X$ is Binomial($n,p$) with mean $p \in (0,1)$.   The distribution of $X$ is unimodal, with mode at $k_0$ equal to $ \lfloor np  \rfloor$ or $\lceil np \rceil$. 
Suppose $Y$ is  Binomial$(n+1,p)$.   Then  $\dtv(X,Y  ) = p \ \p(X=k_0)$.  By Stirling's formula, as $n \to \infty$, $\dtv(X,Y ) \sim p /\sqrt{2 \, \pi \,  n \, p (1-p) }$.
\end{example}

\begin{example} \label{critical x}
 Lemma \ref{Sn is logconcave} states that the distribution of $S_n$ is logconcave,  hence it is unimodal, so Lemma \ref{dtv unimodal shift} and Corollary  \ref{dtv var plus bernoulli}
apply.   We note that in \eqref{S ratio},   the limit as $n \to \infty$ of the ratio is
$$
\frac{\p(S=k+1)}{\p(S=k)} = \frac{x}{1-x^{k+1}}.
$$
For $k=0,1,2,\dots$,  the critical value   $x_k$,  where a tie occurs between $\p(S=k)$ and $\p(S=k+1)$,  is the solution $x_k$ of $x^{k+1}=1-x$.  In particular  $x_0=1/2$,  and $x_1 \doteq .61803$ is one less than the golden mean.
For $x \in (x_k,x_{k+1})$, for all sufficiently large $n$,  the mode of $S_n$ occurs at $k$.

For very large $k$,  $x_k$ is close to 1.  A convenient way to analyze the asymptotic relation is to \emph{define}  $y_k$  by the relation $x_k = \exp(-1/y_k)$;   this leads to
$1/y_k \doteq 1-x_k = (x_k)^{k+1} \doteq (x_k)^k = \exp(-k/y_k)$ so that $y_k \doteq \exp(k/y_k)$
and $k/y_k \doteq \log(y_k)$.  Along these lines it can be proved that as $k \to \infty$,  $k \sim y_k \log y_k$, and an even more careful analysis reveals that
$$
k = y_k\log y_k -\tfrac12 - \tfrac1{24}y^{-1}+O(k^{-2}).
$$
For example,  to have  $x_k \doteq .999$ we consider $y=1000$  with $y \log y 
 \doteq 6907.755$;   the exact values nearby are
$x_{6907} =0.9990004676\dots$    and $x_{6908}=0.9990005939\dots$.
 
\end{example}
 
\begin{proposition}\label{Rn tail}
For all $x\in (0,1)$, and $n\ge1$,  $\p(R_n > 1) \le x^{2n}/(1-x)^2$.
\end{proposition}
\begin{proof}
From Lemma \ref{rdist}, we know the distribution of $R_n=S-S_n$ is given by
$$
\P(R_n=k)=x^{(n+1)k}\frac{\prod_{i=n+1}^{\infty}(1-x^i)}{\prod_{i=1}^k(1-x^i)}.
$$
Therefore,
\begin{align*}
\P(U_n\neq R_n)
	&=1-\P(R_n=0)-\P(R_n=1)\\
	&=1-\prod_{i=n+1}^\infty (1-x^i)-\frac{x^{n+1}}{1-x}\prod_{i=n+1}^\infty (1-x^i)\\
	&\le \frac{x^{n+1}}{1-x}-\frac{x^{n+1}}{1-x}\prod_{i=n+1}^\infty (1-x^i)\\
	&\le \(\frac{x^{n+1}}{1-x}\)^2.
\end{align*}¥
\end{proof}

\begin{theorem}\label{thm dtv sum 2} For all $x\in (0,1)$, as $n \to \infty$,
$$\dtv(S,S_n)\sim C_x\cdot \frac{x^{n+1}}{1-x},$$
where $C_x=\max_{k\in \BZ_+} \P(S=k)$.
\end{theorem}
\begin{proof}
Recall from \eqref{S} and \eqref{R} that $S=S_n+R_n$,  with $S_n$ independent of $R_n$.  Let $U_n:=\min(R_n,1)$, and write $p_n=\p(U_n=1)$;  note that $S_n$ is independent of $U_n$. 
As noted in Example \ref{critical x},  the distribution of $S_n$ is    unimodal,  
so Lemma
\ref{dtv var plus bernoulli} applies, to give
\begin{equation}\label{side BC}
\dtv(S_n,S_n+U_n) = p_n \ \max_k \p(S_n=k).
\end{equation}
Note that $p_n := \p(R_n \ge 1) \ge \p(Z_{n+1} \ge 1) = x^{n+1}$,  and $ \max_k \p(S_n=k) \ge \p(S_n=0)= g_n(x) \ge g(x) >0$,  with $g$ specified by \eqref{g}.   Combined, we have
$\dtv(S_n,S_n+U_n) \ge x^{n+1} g(x)$.

From our particular coupling, we have
\begin{eqnarray*}
  \dtv(S,S_n+U_n) &=   &\dtv(S_n+R_n, S_n+U_n) \\
       & \le & \p(S_n+R_n \ne S_n+U_n) \\
& = & \p(R_n >1). 
\end{eqnarray*}

Now consider the triangle, with vertices at  $A=S$, $B=S_n$, and $C=S_n+U_n$.  The first paragraph of this proof says the the length of side $BC$ is at least $x^{n+1} g(x)$.   The second paragraph says that the length of side $AC$ is at most $\p(R_n>1)$,  which by Proposition \ref{Rn tail} is $O(x^{2n})$,  and since $x \in (0,1)$, the length of $AC$ is little oh of the length of $BC$.   Therefore, by the triangle inequality, the length of $AB$ is asymptotic to the length of $BC$, so from \eqref{side BC},
$$
  \dtv(S,S_n) \sim  p_n \ \max_k \p(S_n=k).
$$

Finally,  $p_n = 1 - \p(R_n=0) = 1 - \prod_{i>n} (1-x^i) \sim x^{n+1}/(1-x)$, as in Theorem \ref{thm dtv process}, and 
 $S_n$ converges to $S$ in distribution, hence, as $n\to\infty$,  $\max_{k} \P(S_n=k)\to \max_{k} \P(S=k) =: C_x$.
\end{proof}

As noted in Example \ref{critical x}, when $x\le 1/2$,  
the mode of $S$ occurs at $0$, so that
$$
\dtv(S,S_n)\sim \frac{g(x)}{1-x}\cdot x^{n+1},
$$
and we see that Theorem \ref{thm dtv sum} is a special case of Theorem \ref{thm dtv sum 2}.

\bibliographystyle{acm}

\bibliography{crank}

\end{document}